\documentclass[11pt]{amsart}
\usepackage[T1]{fontenc}
\usepackage[english]{babel}
\usepackage{kpfonts}

\usepackage{amsmath,amssymb,amsthm,enumerate,xspace}
\usepackage{comment}
\usepackage{accents}
\usepackage[colorlinks=true,citecolor=blue,linkcolor=blue]{hyperref}
\usepackage[all]{xy}
\usepackage{tikz}
\usepackage{mathrsfs}
\usepackage{cleveref}

\numberwithin{equation}{section}

\newtheorem{thm}{Theorem}

\newtheorem*{prob*}{Problem}

\newtheorem{prop}[thm]{Proposition}
\newtheorem{lem}[thm]{Lemma}
\newtheorem{cor}[thm]{Corollary}

\newtheorem*{iprob*}{Problem}

\theoremstyle{definition}

\newtheorem*{defi*}{Definition}

\newtheorem*{acks*}{Acknowledgements}

\newcommand{\ZZ}{\mathbf{Z}}

\newcommand{\RR}{\mathbf{R}}
\newcommand{\NN}{\mathbf{N}}

\newcommand{\hb}{\mathrm{H}_{\mathrm{b}}}

\newcommand{\hh}{\mathrm{H}}

\newcommand{\se}{\subseteq}

\newcommand{\lra}{\longrightarrow}

\title[Lamplighters and the bounded cohomology of Thompson's group]{Lamplighters and the bounded cohomology\\ of Thompson's group}

\author[Nicolas Monod]{Nicolas Monod}
\address{\'Ecole Polytechnique Fédérale de Lausanne (EPFL)\\
CH–1015 Lausanne,
Switzerland}
\begin{document}

\begin{abstract}
We prove the vanishing of the bounded cohomology of lamplighter groups for a wide range of coefficients. This implies the same vanishing for a number of groups with self-similarity properties, such as Thompson's group $F$. In particular, these groups are boundedly acyclic.

Our method is ergodic and applies to ``large'' transformation groups where the Mather--Matsumoto--Morita method sometimes fails because not all are acyclic in the usual sense.
\end{abstract}
\maketitle



\section{Introduction}
The initial goal of this note is to prove the following, answering a question of Grigorchuk~\cite[p.~131]{Grigorchuk95}.

\begin{thm}\label{thm:F:R}
Thompson's group $F$ is boundedly acyclic.
\end{thm}

This statement means that the bounded cohomology $\hb^n(F)$ vanishes for all $n>0$, where  $\hb^n(-)$ denotes the bounded cohomology (of Gromov~\cite{Gromov} and Johnson~\cite[\S2]{Johnson}) with coefficients in $\RR$ viewed as a trivial module. As for the group $F$, it appeared in many different contexts~\cite{Cannon-Floyd-Parry, Cannon-Floyd} since its 1965 definition by Thompson~\cite{Thompson_unpublished} but the main outstanding question seems to be whether it is amenable.

\medskip
One motivation for \Cref{thm:F:R} is that bounded acyclicity is a necessary condition for amenability, although far from sufficient. In fact, amenability is equivalent to the vanishing of $\hb^n(-,E)$ with coefficients in \emph{all} dual Banach modules $E$, see~\cite[Thm.~2.5]{Johnson}. In that context, \Cref{thm:F:R} is a special case of the following.

\begin{thm}\label{thm:F:sep}
The vanishing $\hb^n(F,E)=0$ holds for all $n>0$ and all separable dual Banach $F$-modules $E$.
\end{thm}

This is the first known example of acyclicity for such general coefficients --- except of course amenable groups.

However we caution the reader that this statement does not answer the amenability question. Indeed, our proof also works for many groups that are similar to $F$ but known to be non-amenable. For instance, the proof holds unchanged for all piecewise-projective groups introduced in~\cite{Monod_PNAS}.

In fact, we reduce \Cref{thm:F:sep} to a general result on ``lamplighter'' groups defined as (restricted) wreath products:

\begin{thm}\label{thm:sep}
Let $G$ be any group and consider the wreath product
\[
W = G \wr \ZZ = \Big( \bigoplus_{\ZZ} G \Big) \rtimes \ZZ.
\]
Then $\hb^n(W,E)$ vanishes for all $n>0$ and all separable dual Banach $W$-modules $E$.
\end{thm}

In particular, the case of the trivial module $E=\RR$ shows that $W$ is boundedly acyclic, answering Question~1.8 in~\cite{Loeh_note17}. The vanishing statement of \Cref{thm:sep} fails if we drop either the separability or the duality assumption.

\bigskip

The connection between this general result for lamplighters and the particular case of Thompson's group $F$ comes from the fact that $F$ contains \emph{co-amenable} lamplighter subgroups. This condition will be further discussed in the proof, but for now we point out that this situation is by far not limited to $F$. Indeed, the following is not much more than a reformulation of \Cref{thm:sep}:

\begin{thm}\label{thm:abstract}
Let $G$ be a group and $G_0<G$ a co-amenable subgroup. Suppose that $G$ contains an element $g$ such that the conjugates of $G_0$ by $g^p$ and by $g^q$ commute for all $p\neq q$ in $\ZZ$.

Then $\hb^n(G,E)=0$ holds for all $n>0$ and all separable dual Banach $G$-modules $E$.
\end{thm}

Ignoring for a moment the definition of co-amenability, this result can be applied under general algebraic conditions that we think of as a form of self-similarity:

\begin{cor}\label{cor:alg}
Let $G$ be a group and $G_0<G$ a subgroup with the following two properties:
\begin{enumerate}[(i)]
\item every finite subset of $G$ is contained in some $G$-conjugate of $G_0$,\label{pt:alg:co-a}
\item $G$ contains an element $g$ such that the conjugates of $G_0$ by $g^p$ and by $g^q$ commute for all $p\neq q$ in $\ZZ$. \label{pt:alg:com}
\end{enumerate}
Then $\hb^n(G,E)=0$ holds for all $n>0$ and all separable dual Banach $G$-modules $E$.

This conclusion holds more generally if \eqref{pt:alg:co-a} is required only for finite subsets of the derived subgroup $G'$ of $G$, or of any (fixed) higher derived subgroup $G^{(k)}$ of $G$.
\end{cor}

Beyond $F$, these conditions are satisfied by many generalisations of this group, including all non-amenable groups introduced in~\cite{Monod_PNAS} and their subgroups studied in~\cite{Lodha-Moore}.

\medskip
We obtain perhaps more intuitive conditions with the following special case, stated for groups appearing as transformations of some underlying space. Recall that a transformation is said to be \emph{supported} on a subset if it is the identity outside that subset.

\begin{cor}\label{cor:action}
Let $G$ be a group acting faithfully on a set $Z$. Suppose that $Z$ contains a subset $Z_0$ and that $G$ contains an element $g\in G$ with the following properties:

\begin{enumerate}[(i)]
\item every finite subset of $G$ can be conjugated so that all its elements are supported in $Z_0$;
\item $g^p(Z_0)$ is disjoint from $Z_0$ for every integer $p\geq 1$.
\end{enumerate}

\noindent
Then $\hb^n(G,E)=0$ holds for all $n>0$ and all separable dual Banach $G$-modules $E$. In particular, $G$ is boundedly acyclic.
\end{cor}

Many classical ``large'' transformation groups satisfy these conditions, starting with the group of compactly supported homeomorphisms of $\RR^n$ which was proved to be boundedly acyclic by Matsumoto--Morita~\cite{Matsumoto-Morita}; the latter result was recently widely generalised in~\cite{Loeh_note17} and~\cite{FFLM_binate_draft}. 

\medskip
\Cref{cor:action} has the following advantage in comparison with the strategy behind the classical Matsumoto--Morita~\cite{Matsumoto-Morita} theorem and its generalisations, a strategy rooted in Mather's acyclicity theorem for ordinary (co)homology~\cite{Mather71}. Namely, the latter relies ultimately on pasting together \emph{infinitely many} compactly supported  transformation (which Berrick describes as occasionally ``difficult to substantiate''~\cite[3.1.6]{Berrick02}). In our approach, only finitely many elements need to be pasted together at any one time, according to the definition of lamplighters, recalling the Aristotelian distinction between actual and potential infinity.

This makes it possible to apply \Cref{cor:action} to less flexible situations, such as diffeomorphisms or PL homeomorphisms, see~\cite{Monod-Nariman_arxv2}. By contrast, in ordinary cohomology, the acyclicity of these less flexible groups is often not true or unresolved. We refer to \cite{Monod-Nariman_arxv2} for applications.

In this context, we mention that Kotschick introduced in~\cite{Kotschick08} a commuting conjugates condition which resembles one half of our hypothesis in \Cref{cor:action} and hence holds more generally. He used it to prove the vanishing of the stable commutator length~\cite{Calegari_scl}, which according to Bavard's result~\cite{Bavard91} follows from the vanishing of $\hb^2(-,\RR)$. There is no reason, however, that Kotschick's groups should all be boundedly acyclic or satisfy the vanishing with coefficients. Similar comments hold for the very recent ``commuting conjugates'' criterion for the vanishing of $\hb^2(-,\RR)$ given in~\cite{FFLodha_1_draft}.

\begin{acks*}
I am indebted to S.~Nariman for rekindling my interest in the question of bounded acyclicity. I am grateful to F.~Fournier-Facio, C.~L{\"o}h and M.~Moraschini for showing me their preprint~\cite{FFLM_binate_draft}, for referring me to Grigorchuk's question and for many comments.
\end{acks*}

\section{Lamplighters}\label{sec:}
We shall first work towards the proof of \Cref{thm:sep} for the case of a \emph{countable} group $G$. This restriction will be lifted in Section~\ref{sec:infinity}.

\subsection{Ergodicity}\label{s-sec:erg}
In order to handle non-trivial coefficients, we recall the notion of \emph{ergodicity with coefficients} introduced with M.~Burger in~\cite{Burger-Monod3}. The reader only interested in bounded acyclicity can skip this discussion.

Consider a group $G$ with a non-singular action on a standard probability space $X$. Recall that usual ergodicity is equivalent to the statement that every $G$-invariant measurable function (class) $f\colon X\to\RR$ is (essentially) constant. There is no difference if instead $f\colon X\to E$ ranges in a separable Banach space, or indeed any polish space. However, if $E$ is endowed with a non-trivial $G$-representation, then the requirement that every $G$-\emph{equivariant} measurable function class $f\colon X\to E$ be essentially constant is much stronger. This is called ergodicity with coefficients in $E$. A standard fact is that ergodicity with coefficients in \emph{separable} Banach modules follows from the ergodicity of the diagonal action on $X^2$. It is important here to recall that Banach modules are always assumed to be endowed with an \emph{isometric} $G$-representation. The proof consists in observing that $f\colon X\to E$ must satisfy that $\| f(x) - f(x') \|$ is essentially constant, and that this constant must be zero because $E$ is second countable.

One of the earliest examples of an ergodic action is the \emph{Bernoulli shift}, defined as follows. Let $X$ be a standard probability space and consider the countable power $Y=X^\ZZ$. Then the shift of coordinates is an ergodic $\ZZ$-action on $Y$; this is essentially the same statement as Kolmogorov's zero-one law. It follows that the diagonal action on any power $Y^d$ is also ergodic, because $Y^d$ can be identified with $(X^d)^\ZZ$ in a $\ZZ$-equivariant manner. Considering $2d$ instead of $d$, we can upgrade this to ergodicity with separable coefficients.

\subsection{Amenable actions and co-amenable subgroups}
In the proof of \Cref{thm:sep}, we shall use the terminology of \emph{amenable actions in Zimmer's sense} in the context of a countable group with a non-singular action on a standard probability space. We refer to~\cite{Zimmer78b} or to~\cite[\S4]{Zimmer84} for the definition and recall the following basic facts.

\begin{lem}\label{lem:ZA}
(i)~Consider Zimmer-amenable actions of $G_i$ on $X_i$ for $i=1,2$. Then the product action of $G_1\times G_2$ on $X_1\times X_2$ is Zimmer-amenable.

(ii)~Suppose that $G$ is the union of an increasing sequence of subgroups $G_n<G$, $n\in \NN$. A non-singular action of $G$ is Zimmer-amenable if the corresponding action of every $G_n$ is Zimmer-amenable.
\end{lem}

\begin{proof}
Both statements follow readily from Zimmer's original definition. They also both follow from the equivalent characterisation given in~\cite[Thm.~A(v)]{Adams-Elliott-Giordano}.
\end{proof}

\begin{cor}\label{cor:ZA}
For each integer $n$, let $G_n$ be a countable group with a Zimmer-amenable non-singular action on a standard probability space $X_n$. Then the action of the restricted product $\bigoplus_n G_n$ on the (unrestricted) product $\prod_n X_n$ is a Zimmer-amenable non-singular action on a standard probability space.
\end{cor}

\begin{proof}
  Viewing $\bigoplus_n G_n$ as an increasing union of finite products, this follows from combining the two statements of \Cref{lem:ZA}. Note that the action is non-singular since any given group element acts on only finitely many coordinates.
\end{proof}

We now observe that there is a nice interplay between Zimmer-amenability and Eymard's notion of co-amenability~\cite{Eymard72}. Recall that a subgroup $H<G$ is \emph{co-amenable in $G$} if there is a $G$-invariant mean on $G/H$; another equivalent condition is recalled in the proof below.

\begin{prop}\label{prop:co-ZA}
Let $G$ be a countable group with a non-singular action on a standard probability space $X$. Let $H<G$ be a co-amenable subgroup. If the corresponding $H$-action on $X$ is Zimmer-amenable, then so is the $G$-action.
\end{prop}

\begin{proof}
Recall that a group $G$ is amenable if and only if every convex compact $G$-set $K\neq \varnothing$ (in a Hausdorff locally convex space) has a $G$-fixed point; we can moreover assume that the ambient space is the dual of a Banach $G$-module in the weak-* topology, and for countable $G$ it suffices to consider duals of a separable spaces.

More generally, a subgroup $H<G$ is co-amenable if this condition holds for the subclass of those G-sets $K$ which have an $H$-fixed point.

Finally, Zimmer's definition of the amenability of the $G$-action on $X$ is the fixed-point property for the following specific class of $G$-sets $K$. Start with a measurable \emph{field} of convex compact $K_x\neq \varnothing$ ($x\in X$), but with a cocycle action over $G\times X$ instead of a $G$-action. Then the requirement of Zimmer's definition is a $G$-fixed point in the convex compact $G$-set $K=L^\infty(X, K_\bullet)$ of measurable sections. Since the $G$-action on $K$ combines the $G$-action on $X$ with the cocycle action on  $\{K_x\}$, a section is $G$-fixed exactly when it is cocycle-equivariant as a map; compare~\cite[4.3.1]{Zimmer84}.

Combining these reformulations of the definitions yields the statement of the proposition.
\end{proof}

Since we just recalled the definition(s) of co-amenability, we can see how this notion is relevant to condition~\eqref{pt:alg:co-a} in \Cref{cor:alg}:

\begin{prop}\label{prop:co-a}
Let $G$ be any group and $G_0<G$ a subgroup.

If every finite subset of $G$ is contained in some conjugate of $G_0$, then $G_0$ is co-amenable in $G$.

This holds more generally if every finite subset of a fixed co-amenable subgroup $G_1<G$ is contained in some $G$-conjugate of $G_0$.
\end{prop}

The more general form above is relevant to the additional statement in \Cref{cor:alg} by setting $G_1=G^{(k)}$; indeed $G^{(k)}$ is co-amenable in $G$ since it is normal with soluble quotient.

\begin{proof}[Proof of \Cref{prop:co-a}]
Let $K$ be a convex compact $G$-set containing some $G_0$-fixed point $k$. We need to prove that $K$ has a $G$-fixed point, but it suffices to show that it has a $G_1$-fixed point since the latter is co-amenable in $G$. Given any finite subset $F$ of $G_1$, let $g_F\in G$ be an element conjugating $F$ into $G_0$. Then $g_F k$ is fixed by the group generated by $F$. Consider $g_F k$ as a net indexed by the directed set of all finite subsets $F$ of $G_1$. Then any accumulation point of this net in the compact space $K$ will be fixed by $G_1$.
\end{proof}

The connection between co-amenability and bounded cohomology is the following basic fact. A proof can be found e.g. in~\cite[Prop.~8.6.6]{Monod} and actually it provides a characterisation of co-amenability, see~\cite[Prop.~3]{Monod-Popa}. (There is a countability assumption in~\cite[Prop.~8.6.6]{Monod} but it is not needed nor used.)

\begin{prop}\label{prop:rest:co-a}
Let $G$ be a group, $G_0<G$ a co-amenable subgroup and $E$ a dual Banach $G$-module. Then the restriction map
\[
\hb^\bullet(G, E) \lra \hb^\bullet(G_0, E)
\]
is injective.\qed
\end{prop}

\subsection{Vanishing for countable wreath products}\label{s-sec:wreath}
Let $G$ be a countable group, $W=G \wr \ZZ$ the (restricted) wreath product and $E$ a separable dual $W$-module.

Let further $X$ be any standard probability space with a non-singular $G$-action which is amenable in Zimmer's sense. One can for instance simply take $X$ to be $G$ itself, endowed with any distribution of full support. The countable power $Y=X^\ZZ$ is a standard probability space with a non-singular action of $\bigoplus_{n\in \ZZ} G$ which is Zimmer-amenable by \Cref{cor:ZA}. We let $\ZZ$ act on $Y$ by shifting the coordinates; that is, $Y$ is the $X$-based Bernoulli shift. Combining the two actions, we have thus endowed $Y$ with an action of the (restricted) wreath product $W=G \wr \ZZ$.

The subgroup $\bigoplus_{n\in \ZZ} G$ is co-amenable in $W$ since it is a normal subgroup with amenable quotient. Therefore, the above discussion allows us to apply \Cref{prop:co-ZA} and conclude that the $W$-action on $Y$ is Zimmer-amenable.

According to~\cite[Thm.~2]{Burger-Monod3} or to~\cite[Thm.~7.5.3]{Monod}, the bounded cohomology of $W$ with coefficient in $E$ is realised by the complex of $W$-equivariant measurable bounded function classes
\[
0 \lra L^\infty(Y, E)^W \lra L^\infty (Y^2, E)^W  \lra L^\infty (Y^3, E)^W \lra \cdots
\]
with the usual ``simplicial'' Alexander--Kolmogorov--Spanier differentials. (The general references above specify that measurability is in the weak-* sense, but this is irrelevant here, see~\cite[Lem.~3.3.3]{Monod}.)

The discussion of Section~\ref{s-sec:erg} shows that the $\ZZ$-equivariant elements of  $L^\infty (Y^d, E)$ are essentially constant, and hence so is every element of $L^\infty (Y^d, E)^W$. It follows that the latter space consists of all essentially constant maps ranging in $E^W$. Since the simplicial differentials are alternating sums of the map omitting each variable, we conclude that the above complex of $W$-equivariant maps is acyclic except possibly in degree zero, where its cohomology is $E^W$. This completes the proof of \Cref{thm:sep} when $G$ is countable.

\subsection{To \texorpdfstring{$\aleph_0$}{infinity} and beyond}\label{sec:infinity}
We imposed a countability assumption on our groups in order to be able to apply standard ergodic methods straight out of the shipping box. There does not appear to be a deeper reason that countability should be needed. In any case, the following general principle will allow us to reduce \Cref{thm:sep} to the countable case.

\begin{prop}\label{prop:red:countable}
Let $G$ be any group and $E$ a separable dual Banach $G$-module. Suppose that every countable subset of $G$ is contained in a subgroup $G_1 < G$ such that $\hb^n(G_1, E)$ vanishes for all $n>0$.

Then $\hb^n(G, E)$ vanishes for all $n>0$.
\end{prop}

Since the classical examples of boundedly acyclic groups were precisely large, uncountable, groups, we single out the following particular case of \Cref{prop:red:countable}.

\begin{cor}\label{cor:red:countable}
Let $G$ be any group. Suppose that every countable subset of $G$ is contained in some boundedly acyclic subgroup of $G$.

Then $G$ is boundedly acyclic.\qed
\end{cor}

\begin{proof}[Proof of \Cref{prop:red:countable}]
Consider a group $J$ and a separable Banach $J$-module $E$ which is the dual of some Banach $J$-module $F$. The key claim is that $\hb^n(J, E)$ vanishes for all $n>0$ if and only if the $\ell^1$-homology $\hh^{\ell^1}_n(J,F)$ vanishes for all $n>0$. This fact then implies the proposition when we apply it to both $G$ and $G_1$, viewing $E$ also as a dual Banach $G_1$-module.

The claim is established as~\cite[Cor.~2.4(iii)]{Matsumoto-Morita} in the special case of $E$ trivial, which is sufficient for \Cref{cor:red:countable}.

In the general case, we recall first that $\hb^n(J, E)$ vanishes for all $n>0$ if and only if the following two conditions hold: first, $\hh^{\ell^1}_n(J,F)$ vanishes for all $n>0$; second, $\hh^{\ell^1}_0(J,F)$ is Hausdorff. This equivalence relies on the closed range theorem and was already established by Johnson~\cite[Cor.~1.3]{Johnson}, see also~\cite[Cor.~2.4(i),(ii)]{Matsumoto-Morita}. Moreover, the second condition is equivalent to $\hb^1(J, E)$ being Hausdorff, see~\cite[Thm.~2.3]{Matsumoto-Morita}

In conclusion, it certainly suffices to prove $\hb^1(J, E)=0$. This amounts to showing that every affine isometric action with bounded orbits on a separable dual Banach space admits a fixed point. That statement is a variant of the Ryll-Nardzewski theorem. Specifically, it is the case (c) in Bourbaki's account, Appendix~3 to part~IV of~\cite{BourbakiTVS}. Warning: the English translation incorrectly requires first countability for the norm of $E$, which is both empty and insufficient, whereas the proof and the French original correctly use second countability, which here is just the separability of $E$ that we assumed.
\end{proof}

\begin{proof}[End of proof of \Cref{thm:sep}]
Let $G$ by an arbitrary group and $S\se G\wr \ZZ$ a countable subset of the wreath product. Then there is a countable subgroup $G_1$ of $G$ such that $G_1\wr \ZZ$ contains $S$. In consequence, \Cref{prop:red:countable} allows us to reduce the general case to the case of wreath products of countable groups.
\end{proof}

\section{Self-similar situations}
\begin{flushright}
\begin{minipage}[t]{0.85\linewidth}\itshape\small
\begin{flushright}
It is shaped, sir, like itself; and it is as broad as it hath breadth: it is just so high as it is, and moves with its own organs
\end{flushright}
\begin{flushright}
\upshape\small
Shakespeare, \emph{Antony and Cleopatra}, Act 2, Scene 7
\end{flushright}
\end{minipage}
\end{flushright}

\medskip
We now consider the situation where a group contains a suitable supply of copies of a given co-amenable subgroup; in particular the case of Corollaries~\ref{cor:alg} and~\ref{cor:action} where the replicating subgroup also progressively swallows the entire ambient group.

\begin{proof}[Proof of \Cref{thm:abstract}.]
Let $G$ be a group, $G_0<G$ a co-amenable subgroup and suppose that $G$ contains an element $g$ such that the conjugates of $G_0$ by $g^p$ and by $g^q$ commute for all $p\neq q$ in $\ZZ$. Let further $E$ be a separable dual Banach $G$-module.

Let $W_1<G$ be the subgroup generated by $G_0$ and $g$. Then $W_1$ is co-amenable in $G$ since it contains $G_0$. Therefore, by \Cref{prop:rest:co-a}, it suffices to show that $\hb^n(W_1,E)$ vanishes for all $n>0$. The commutativity assumption implies that there is a natural map from the wreath product $W = G_0 \wr \ZZ$ onto $W_1$ mapping $1\in \ZZ$ to $g$. This turns $E$ into a $W$-module and \Cref{thm:sep} implies the vanishing of $\hb^n(W,E)$. It only remains to justify that the inflation
\[
\hb^n(W_1,E) \lra \hb^n(W,E)
\]
is injective. By construction, the kernel of the projection $W \to W_1$ is metabelian and hence amenable. This implies that the inflation is an isomorphism (see~\cite[Thm.~1]{Noskov} or~\cite[Rem.~8.5.4]{Monod}) and thus \Cref{thm:abstract} follows.
\end{proof}

\begin{proof}[Proof of \Cref{cor:alg}]
Let $G$ be a group, $E$ a separable dual Banach $G$-module and $G_0<G$ a subgroup. To prove the general case of \Cref{cor:alg}, we can assume the following.

\begin{enumerate}[(i)]
\item Every finite subset of a (fixed) co-amenable subgroup $G_1<G$ is contained in some $G$-conjugate of $G_0$.\label{pt:alg:co-a:proof}
\item $G$ contains an element $g$ such that the conjugates of $G_0$ by $g^p$ and by $g^q$ commute for all $p\neq q$ in $\ZZ$.\label{pt:alg:com:proof}
\end{enumerate}

\noindent
In view of \Cref{prop:co-a}, the first condition implies that $G_0$ is co-amenable in $G$. Therefore, we are in a position to apply \Cref{thm:abstract} and hence \Cref{cor:alg} is established.
\end{proof}

\begin{proof}[Proof of \Cref{cor:action}]
We are given a group $G$ acting faithfully on a set $Z$, an element $g\in G$ and a subset $Z_0\se Z$ such that:

\begin{enumerate}[(i)]
\item every finite subset of $G$ can be conjugated so that all its elements are supported in $Z_0$;
\item $g^p(Z_0)$ is disjoint from $Z_0$ for every integer $p\geq 1$.
\end{enumerate}

Define $G_0<G$ to be the subgroup consisting of all elements of $G$ supported in $Z_0$. The first assumption implies that every finite subset of $G$ is contained in some conjugate of $G_0$. Next, note that the sets $g^p(Z_0)$ and $g^q(Z_0)$ are disjoint for all distinct $p,q\in \ZZ$. Since $G$ acts faithfully, this implies that the conjugates of $G_0$ by $g^p$ and by $g^q$ commute. Thus \Cref{cor:action} follows indeed from \Cref{cor:alg}.
\end{proof}

We can now apply this result to the case of Thompson's group $F$. Since this group has a number of very different descriptions linking it to interesting objects in homotopy, algebra and combinatorics, we should specify which description of $F$ we work with. We consider $F$ to be the group of piecewise affine homeomorphisms of $[0,1]$ with dyadic breakpoints and slopes in $2^\ZZ$. We refer to~\cite{Cannon-Floyd-Parry} for background.

\begin{proof}[Proof of \Cref{thm:F:sep} and hence also of \Cref{thm:F:R}]
We work with the derived subgroup $F'$ of $F$, which is sufficient by \Cref{prop:rest:co-a}. Choose a non-trivial element $g\in F'$ and choose a dyadic point $x_0 \in (0,1)$ not fixed by $g$. Let $Z$ be the $F'$-orbit of $x_0$ (which happens to consist of all dyadic points of $(0,1)$). Define $Z_0$ to be the open interval determined by $x_0$ and $g(x_0)$, which is non-empty by construction.

The first condition of \Cref{cor:action} is satisfied by the transitivity properties of the $F'$-action: any interval stricly contained in $(0,1)$ can be shrunk into $Z_0$ (compare e.g.\ the proof of Cor.~2.3 in~\cite{Caprace-Monod_discrete}). The second condition follows from the fact that $g$ preserves the order on $(0,1)$. In conclusion, we are in a situation to appeal to \Cref{cor:action} and thus complete the proof of \Cref{thm:F:sep} and hence also of \Cref{thm:F:R}.
\end{proof}

(The reader might notice that this reasoning, when brought all the way back to the underlying wreath product subgroup of $F$, is an improvement of our comments in Section~6.C of~\cite{MonodVT}.)

\medskip
The above argument holds for many similar groups since it only relies on the abstract statement of \Cref{cor:action}; as mentioned in the introduction, this includes all piecewise-projective groups of homeomorphisms of the line that have sufficiently transitive orbits.

\section{Further comments}\label{sec:further}
\subsection{More acyclicity}
First, we should recall that many examples of boundedly acyclic groups (with trivial coefficients) have been discovered, starting with the theorem of Matsumoto--Morita~\cite{Matsumoto-Morita}. Recent examples include, among others, \cite{Loeh_note17}, \cite{FFLM_binate_draft}, \cite{FFLM_2106.13567}, \cite{Monod-Nariman_arxv2}. A very nice general criterion, but for degree two only, is given in~\cite{FFLodha_1_draft}.

\medskip

Furthermore, boundedly acyclic groups can be used as a tool in results aiming to determine \emph{non-trivial} bounded cohomology of larger groups. This has recently led to the complete computation of the bounded cohomology of some groups that are not boundedly acyclic~\cite{Monod-Nariman_arxv2}. These methods, or the methods of~\cite[\S6]{FFLM_binate_draft}, can now leverage \Cref{thm:F:R} above to establish that the bounded cohomology of Thompson's circle group $T$ is generated by the bounded Euler class, in perfect analogy to the result of~\cite{Monod-Nariman_arxv2} about the entire group of (orientation-preserving) homeomorphisms of the circle. By contrast, the usual cohomology of $T$ (and of $F$) is richer and completely described in~\cite{Ghys-Sergiescu}.

\subsection{More lamplighters}
In the proof of the vanishing for wreath products, the only properties of $\ZZ$ that we used were that it is infinite and amenable. Therefore, the same results holds more generally for all (restricted) wreath products $G\wr \Gamma$ as long as $\Gamma$ is infinite amenable.

\medskip
A closer examination of the proof also shows that it holds for suitable permutational wreath products where $\ZZ$ is replaced by an amenable group with a permutation action having only infinite orbits; the ergodicity of the corresponding generalised Bernoulli shifts is recorded e.g. in~\cite[Prop.~2.1]{Kechris-Tsankov}.

\medskip
These generalisations can in turn be used to extend the statements of \Cref{thm:abstract} and of \Cref{cor:alg}. Specifically, instead of a single element $g$ and the corresponding commuting conjugates $G_0^{g^p}$, it suffices to assume there is some infinite amenable subgroup $\Gamma<G$ such that the $\Gamma$-conjugates of $G_0$ commute pairwise. We can proceed similarly to generalise \Cref{cor:action}.

\subsection{No more coefficients}
As recalled in the introduction, \Cref{thm:sep} does not hold without the separability assumption on the dual module $E$. Indeed, in that case the vanishing of $\hb^n(-, E)$ for all $n>0$ is equivalent to amenability.  (\Cref{thm:sep} does hold for some very specific non-separable modules such as the \emph{semi-separable} case introduced in~\cite{MonodVT}, but only because they can be reduced to the separable case.)

A very concrete example, not relying on the huge coefficient module witnessing non-amenability in general, is as follows. Choose a group $G$ with $\hb^2(G)\neq 0$, for instance a free group of rank two. Then, by inflation, $\hb^2\big(\bigoplus_\ZZ G \big)$ is also non-zero. It follows by cohomological induction~\cite[\S10.1]{Monod} that
\[
\hb^2\big( G\wr \ZZ, \ell^\infty(\ZZ) \big) \ \neq 0,
\]
where $G\wr \ZZ$ acts on $\ell^\infty(\ZZ)$ by translation via the quotient morphism $G\wr \ZZ \to\ZZ$. The only circumstance preventing a contradiction with \Cref{thm:sep} is that the dual Banach module $\ell^\infty(\ZZ)$ is not separable.

\medskip
We mention here that Grigorchuk asked about the vanishing of $\hb^n(F, E)$ for all $n\geq 2$ and all dual $E$, see Problem~3.19 in~\cite{Grigorchuk95}. This is a priori a weakening of amenability, first considered by Johnson in \S10.10 of~\cite{Johnson}. However, it was shown in~\cite[Cor.~5.10]{MonodICM} that this condition is in fact also equivalent to amenability; this result relies on the Gaboriau--Lyons theorem~\cite{Gaboriau-Lyons}.

\medskip
The condition that $E$ be \emph{dual} cannot be removed either (even when keeping it separable). Indeed it is well-known that the vanishing of $\hb^1(G, E)$ for all Banach modules characterises \emph{finite} groups $G$. This follows by applying the cohomological long exact sequence (for bounded cohomology) to the submodule inclusion $\ell^1_0 (G) \to \ell^1(G)$, where  $\ell^1_0 (G)$ denotes the summable functions with vanishing sum. This module is separable when $G$ is countable.

\subsection{Amenability vs. Ergodicity}
It is a remarkable fact (in the author's opinion) that \emph{every} group admits a Zimmer-amenable space $X$ which is doubly ergodic with separable coefficients. An early proof is found in~\cite{Burger-Monod1,Burger-Monod3} and the most luminous argument is in~\cite{Kaimanovich03}. This cannot be extended to higher ergodicity in general precisely because non-trivial bounded cohomology provides an obstruction.

This fact, together with the observation that the amenability of $F$ is equivalent to the Zimmer-amenability of the $T$-action on the circle, has prompted us to ask whether non-amenable groups can have  Zimmer-amenable actions that are multiply ergodic far beyond two factors, see Problem~H in~\cite{MonodICM}. The construction of Section~\ref{s-sec:wreath} shows that this is indeed possible.


\bibliographystyle{amsalpha}
\bibliography{../BIB/ma_bib}

\end{document}